\documentclass[12pt,reqno]{amsart}
\usepackage{amsfonts}
\usepackage{amsmath}
\usepackage{amssymb}
\usepackage{amscd}
\usepackage{graphicx}
\usepackage{hyperref}

\newtheorem{theorem}{Theorem}[section]

\newtheorem{proposition}{Proposition}[section]
\newtheorem{corollary}{Corollary}[section]

\newtheorem{example}{Example}[section]

\numberwithin{equation}{section}
\begin{document}
\title[Finite Blaschke Products and
Decomposition]{Finite Blaschke Products and
Decomposition}
\author{S\"{u}meyra U\c{C}AR$^{1}$}
\address{Bal\i kesir University\\
Department of Mathematics\\
10145 Bal\i kesir, TURKEY}
\email{sumeyraucar@balikesir.edu.tr}
\author{Nihal YILMAZ\ \"{O}ZG\"{U}R$^{1}$}
\address{Bal\i kesir University\\
Department of Mathematics\\
10145 Bal\i kesir, TURKEY}
\email{nihal@balikesir.edu.tr}
\thanks{$^{1}$Bal\i kesir University, Department of Mathematics, 10145 Bal\i
kesir, TURKEY}

\begin{abstract}
Let $B(z)$ be a finite Blaschke product of degree $n$. We consider the
problem when a finite Blaschke product can be written as a composition of
two nontrivial Blaschke products of lower degree related to the condition $%
B\circ M=B$ where $M$ is a M\"{o}bius transformation from the unit disk onto
itself.
\end{abstract}

\maketitle

\section{\textbf{Introduction}}

\label{intro}

It is known that a M\"{o}bius transformation from the
unit disc $\mathbb{D}$ into itself is of the following form:
\begin{equation}
M(z)=c\frac{z-\alpha }{1-\overline{\alpha }z},  \label{Mobius}
\end{equation}%
where $\alpha \in \mathbb{D}$ and $c$ is a complex constant of modulus one (see \cite{ford} and \cite{gassier}).

The rational function
\begin{equation*}
B(z)=c\underset{k=1}{\overset{n}{\prod }}\frac{z-a_{k}}{1-\overline{a_{k}}z}
\end{equation*}%
is called a finite Blaschke product of degree $n$ for the unit disc where $%
\left\vert c\right\vert =1$ and $\left\vert a_{k}\right\vert <1$, $1\leq
k\leq n$.

Blaschke products of the following form are called canonical Blaschke
products:%
\begin{equation}
B(z)=z\underset{k=1}{\overset{n-1}{\prod }}\frac{z-a_{k}}{1-\overline{a_{k}}z%
},\left\vert a_{k}\right\vert <1\text{ for }1\leq k\leq n-1.
\label{canonical}
\end{equation}

Decomposition of finite Blaschke products is an interesting matter studied
by many researcher by the use of a point $\lambda $ on the unit circle $\partial
\mathbb{D}$ and the points $z_{1},z_{2},...,z_{n}$ on the unit circle $%
\partial \mathbb{D}$ such that $B\left( z_{1}\right) =...=B\left(
z_{n}\right) =\lambda $. For example in $2011,$ using circles passing
through the origin, it was given the determination of these points for the
Blaschke products written as composition of two nontrivial Blaschke products
of lower degree $($see \cite{NYO1} and \cite{NYO2}$)$. On the other hand,
decomposition of finite Blaschke products are related to the condition $%
B\circ M=B$ where $M$ is a M\"{o}bius transformation of the form $($\ref%
{Mobius}$)$ and different from the identity $($see \cite{craighead} and \cite%
{gassier}$)$.

In this paper we consider the relationship between the following two
questions for a given canonical finite Blaschke product:

$Q1)$ Is there a M\"{o}bius transformation $B\circ M=B$ such that $M$ is
different from the identity?

$Q2)$ Can $B$ be written as a composition $B=B_{2}\circ B_{1}$ of two finite
Blaschke products of lower degree?

In Section \ref{sec:1}, we recall some necessary theorems about these
questions. In Section \ref{sec:2} we give some theorems and examples related
to the above two questions.

\section{Preliminaries}

\label{sec:1}

In this section we give some information about decomposition of finite
Blaschke products written as $B\circ M=B$ where $M$ is a M\"{o}bius
transformation different from the identity. In \cite{gassier}, it was proved
that the set of continuous functions $M$ from the unit disc into the unit
disc such that $B\circ M=B$ is a cyclic group if $B$ is a finite Blaschke
product. In \cite{craighead}, the condition $B\circ M=B$ was used in the
following theorem.

\begin{theorem}
\label{thm1}$($See \cite{craighead}, Theorem $3.1$ on page $335)$ Let $B$ be
a finite Blaschke product of degree $n.$ Suppose $M\neq I$ is holomorphic
from $\mathbb{D}$ into $\mathbb{D}$ such that $B\circ M=B.$ Then$:$

$(i)$ $M$ is a M\"{o}bius transformation,

$(ii)$ There is a positive integer $k\geq 2$ such that the iterates, $%
M,...M^{[k-1]}$ are all distinct but $M^{[k]}=I.$

$(iii)$ $k$ divides $n$.

$(iv)$ There is a $\gamma \in \mathbb{D}$ such that $M(\gamma )=\gamma .$

$(v)$ $B$ can be written as a composition $B=B_{2}\circ B_{1}$ of finite
Blaschke products with the degree $B_{1}=k$ and the degree of $B_{2}=n/k.$ $%
B_{1}$ may be taken to be
\begin{equation*}
B_{1}(z)=\left( \frac{z-\gamma }{1-\overline{\gamma }z}\right) ^{k}.
\end{equation*}
\end{theorem}

But the condition $B\circ M=B$ is not necessary for a decomposition of
finite Blaschke products $($see \cite{craighead} for more details$)$.

It follows from Theorem \ref{thm1} that if a finite Blaschke product $B$ can
be written as $B\circ M=B,$ then $B$ can be decomposed into a composition of
two finite Blaschke products of lower order. However the following theorem
gives necessary and sufficient conditions for the question $Q1.$

\begin{theorem}
\label{thm4}$($See \cite{gassier}, Proposition $4.1$ on page $202)$ Let $B$
be finite Blaschke product of degree $n\geq 1,$ $B(z)=\underset{k=1}{\overset%
{n}{\prod }}\frac{z-a_{k}}{1-\overline{a_{k}}z}$ with $a_{k}\in \mathbb{D}$
for $1\leq k\leq n.$ Let $M$ be a M\"{o}bius transformation from $\mathbb{D}$
into $\mathbb{D}$. The following assertions are equivalent$:$

$(i)$ $(B\circ M)(z)=B(z),z\in \mathbb{C},\left\vert z\right\vert \leq 1.$

$(ii)$ $M(\{a_{1},a_{2,}...,a_{n}\})=\{a_{1},a_{2,}...,a_{n}\}$ and there
exists $z_{0}\in \overline{\mathbb{D}}\backslash \{a_{1},a_{2,}...,a_{n}\}$
such that $(B\circ M)(z_{0})=B(z_{0}).$
\end{theorem}

Using the following proposition given in \cite{gassier}, we know how to
construct a finite Blaschke product of degree $n\geq 1$ satisfying the
condition $B\circ M=B$ where M is a M\"{o}bius transformation from $\mathbb{D}$
into $\mathbb{D}$ different from identity.

\begin{proposition}
\label{prop4}$($See \cite{gassier}, Proposition $4.2$ on page $203)$ Let $n$
be a positive integer and let $M$ be a M\"{o}bius transformation from $%
\mathbb{D}$ into $\mathbb{D}$ such that $M^{n}(0)=0$ and $%
\{0,M(0),...,M^{n-1}(0)\}$ is a set of n distinct points in $\mathbb{D}$.
Consider the finite Blaschke product $B(z)=z$ $\underset{k=1}{\overset{n-1}{%
\prod }}\frac{z-M^{k}\left( 0\right) }{1-\overline{M^{k}\left( 0\right) }z}$%
. Then the group $G$ of the invariants of $B$ is generated by $M$.
\end{proposition}

Let $p$ be a prime number. Using Theorem \ref{thm1}, it is clear that for a
Blaschke product of degree $p$, it is impossible $B=B_{2}\circ B_{1}.$

From \cite{horwitz}, we know the following theorem and we will use this
theorem in the next chapter.

\begin{theorem}
\label{thm7}Let%
\begin{equation*}
A\left( z\right) =\underset{k=1}{\overset{n}{\prod }}\frac{z-a_{k}}{1-%
\overline{a_{k}}z}\text{ and }B\left( z\right) =\underset{k=1}{\overset{n}{%
\prod }}\frac{z-b_{k}}{1-\overline{b_{k}}z}
\end{equation*}%
with $a_{k}$ and $b_{k}\epsilon \mathbb{D=\{}\left\vert z\right\vert <1%
\mathbb{\}}$ for $j=1,2,...,n.$ Suppose that $A\left( \lambda _{j}\right)
=B\left( \lambda _{j}\right) $ for $n$ distinct points $\lambda
_{1},...,\lambda _{n}$ in $\mathbb{D}$. Then $A\equiv B.$
\end{theorem}

\section{\textbf{Blaschke Products of Degree }$\mathbf{n}$}

\label{sec:2}

Let $B$ be a canonical Blaschke product of degree $n$ and following \cite%
{gassier}, let $Z(B)$ denotes the elements $z\in \mathbb{D}$ such that $%
B(z)=0.$ In this section we consider the relationship between the questions
(Q1) and (Q2).

Now we give the following theorem for the Blaschke products of degree $n$.

\begin{theorem}
\label{thm6}Let $M\left( z\right) =c\frac{z-\alpha }{1-\overline{\alpha }z}$
be a M\"{o}bius transformation different from the identity from the unit
disc into itself and $B(z)=z\underset{k=1}{\overset{n-1}{\prod }}\frac{%
z-a_{k}}{1-\overline{a_{k}}z}$ be a canonical Blaschke product of degree $n$%
. Then $B\circ M=B$ if and only if $M(z)=c\frac{z-a_{j}}{1-\overline{a_{j}}z}
$ with $\left\vert a_{j}\right\vert =\left\vert a_{l}\right\vert $ for some $%
a_{j},a_{l}$ $(0\leq j,l\leq n-1),$ $c=-\frac{a_{j}}{a_{l}}$ and the equation%
\begin{equation}
M^{n-1}\left( 0\right) -a_{i}=0\text{, }\left( 1\leq i\leq n-1\right)
\label{eqn6}
\end{equation}%
is satisfied by the non-zero zeros of $B$.
\end{theorem}

\begin{proof}
$(\Rightarrow ):$ Let $M\left( z\right) =c\frac{z-\alpha }{1-\overline{%
\alpha }z}$ be a M\"{o}bius transformation different from the identity from
the unit disc into itself, $B(z)=z\underset{k=1}{\overset{n-1}{\prod }}\frac{%
z-a_{k}}{1-\overline{a_{k}}z}$ be a canonical Blaschke product of degree $n$
and $B\circ M=B.$ From Proposition \ref{prop4}, we know $Z(B)=%
\{0,M(0),...,M^{n-1}(0)\}$ and $M^{n}(0)=0$. Without loss of generality, let
us take $a_{1}=M(0)$. Then we find%
\begin{equation}
a_{1}=-c\alpha  \label{eqn10}
\end{equation}%
Let $a_{j}=M^{j}\left( 0\right) $ $\left( 2\leq j-1\leq n-1\right) $ and
then we find the following equations:%
\begin{equation*}
a_{2}=M^{2}\left( 0\right) =-\frac{c\alpha \left( 1+c\right) }{1+c\left\vert
\alpha \right\vert ^{2}},\text{ }a_{3}=M^{3}\left( 0\right)
\end{equation*}%
\begin{equation*}
\vdots
\end{equation*}%
\begin{equation}
a_{n-1}=M^{n-1}\left( 0\right) .  \label{eqn11}
\end{equation}%
By Theorem \ref{thm1}, since the degree of $M$ divides the degree of $B$,
then it should be $M^{n}(0)=0.$ Using the equation $($\ref{eqn11}$)$ we have%
\begin{equation*}
M^{n}\left( 0\right) =M\left( M^{n-1}\left( 0\right) \right) =M\left(
a_{n-1}\right) =0
\end{equation*}%
and so we get%
\begin{equation*}
a_{n-1}=\alpha .
\end{equation*}

By the equation $($\ref{eqn10}$)$ we find $c=-\frac{a_{1}}{a_{n-1}}$ and
hence $\left\vert a_{1}\right\vert =\left\vert a_{n-1}\right\vert $. If we
take $a_{1}=a_{j}$ and $a_{n-1}=a_{l}$ then the proof follows.

$(\Longleftarrow ):$ For the points $0$ and $a_{k}$ $(1\leq k\leq n-1)$ in $%
\mathbb{D}$, we have
\begin{equation*}
\left( B\circ M\right) \left( 0\right) =B\left( 0\right) \text{ and }\left(
B\circ M\right) \left( a_{k}\right) =B\left( a_{k}\right) ,
\end{equation*}%
by the equation $($\ref{eqn6}$).$ Then, by Theorem \ref{thm7} we obtain%
\begin{equation*}
B\circ M\equiv B.
\end{equation*}
\end{proof}

From \cite{gassier}, we know the following corollary for the Blaschke
products of degree $3$ such that $B\circ M=B.$

\begin{corollary}
\label{corol3}$($See \cite{gassier}, page $205)$ Let $G$ be a cyclic group
which is composed of the transformations $M$ such that $B\circ M=B$. Then we
have the following assertions$:$

$(i)$ If $B(z)=z^{3},$ $G$ is generated by $M(z)=e^{2i\pi /3}z,$ $z\in
\overline{\mathbb{D}}.$

$(ii)$ If $Z(B)$ contains a non-zero point in $\mathbb{D}$, $B(z)=z\frac{%
z-a_{1}}{1-\overline{a_{1}}z}\frac{z+\overline{c}a_{1}}{1+c\overline{a_{1}}z}
$ where $\alpha \in \mathbb{D}\backslash \{0\}$ and where $c+\overline{c}%
=-1-\left\vert \alpha \right\vert ^{2}.$ In this case the group $G$ is
generated by $M(z)=c\frac{z+\overline{c}a_{1}}{1+c\overline{a_{1}}z}.$
\end{corollary}

However, as an application of Theorem \ref{thm6}, we give the following
corollary in our form for degree $3$.

\begin{corollary}
\label{corol4}Let $M\left( z\right) =c\frac{z-\alpha }{1-\overline{\alpha }z}
$ be a M\"{o}bius transformation different from the identity from the unit
disc into itself and $B(z)=z\frac{z-a}{1-\overline{a}z}\frac{z-b}{1-%
\overline{b}z}$ be a Blaschke product of degree $3$. Then $B\circ M=B$ if
and only if $M(z)=c\frac{z-b}{1-\overline{b}z}$ with $\left\vert
a\right\vert =\left\vert b\right\vert $ and $c$ is a root of the equation $%
c^{2}+c(1+\left\vert a\right\vert ^{2})+1=0.$
\end{corollary}

As an other application of Theorem \ref{thm6}, a similar corollary can be given for the Blaschke products of prime
degrees. We give the following corollaries and examples for degree $5$ and $7$.

\begin{corollary}
Let $M\left( z\right) =c\frac{z-\alpha }{1-\overline{\alpha }z}$ be a M\"{o}%
bius transformation different from the identity from the unit disc into
itself and $B(z)=z\underset{k=1}{\overset{4}{\prod }}\frac{z-a_{k}}{1-%
\overline{a_{k}}z}$ be a Blaschke product of degree $5$. Then $B\circ M=B$
if and only if $M(z)=c\frac{z-a_{l}}{1-\overline{a_{l}}z}$ with $\left\vert
a_{j}\right\vert =\left\vert a_{l}\right\vert $ for some $(0\leq j,l\leq 4),$
$c=-\frac{a_{j}}{a_{l}}$ and the equation%
\begin{equation}
4c^{2}a_{l}\left\vert a_{l}\right\vert ^{2}+3ca_{l}\left\vert
a_{l}\right\vert ^{2}+3c^{3}a_{l}\left\vert a_{l}\right\vert
^{2}+c^{4}a_{l}+c^{3}a_{l}+c^{2}a_{l}+ca_{l}+a_{l}+c^{2}a_{l}\left\vert
a_{l}\right\vert ^{4}=0,  \label{eqn12}
\end{equation}%
is satisfied by the non-zero zeros of $B$.
\end{corollary}

\begin{example}
Let $B(z)=z\underset{k=1}{\overset{4}{\prod }}\frac{z-a_{k}}{1-\overline{%
a_{k}}z}$ be a Blaschke product of degree $5.$ From Proposition \ref{prop4},
we know that $Z(B)=\{0,M(0),...,M^{n-1}(0)\},$ so we can take%
\begin{eqnarray}
a_{1} &=&M\left( 0\right) ,  \label{eqn13} \\
a_{2} &=&M^{2}\left( 0\right) ,  \notag \\
a_{3} &=&M^{3}\left( 0\right) ,  \notag \\
a_{4} &=&M^{4}\left( 0\right) .  \notag
\end{eqnarray}%
Let $a_{l}=\frac{1}{2}.$ Using the equation $($\ref{eqn12}$),$ we obtain $%
c=-0.856763-i0.515711$ and%
\begin{equation*}
M\left( z\right) =\frac{\left( 0.856763+i0.515711\right) \left( 1-2z\right)
}{2-z}.\text{ }
\end{equation*}%
Using the equation $($\ref{eqn13}$),$ we find%
\begin{eqnarray*}
a_{1} &=&0.428381+i0.257855, \\
a_{2} &=&0.278236-i0.188486, \\
a_{3} &=&0.141178+0.304977i, \\
a_{4} &=&0.5.
\end{eqnarray*}%
Then we find $\left( B\circ M\right) \left( z\right) =B\left( z\right) $ for
the points $z\in \mathbb{D}.$
\end{example}

\begin{corollary}
Let $M\left( z\right) =c\frac{z-\alpha }{1-\overline{\alpha }z}$ be a M\"{o}%
bius transformation different from the identity from the unit disc into
itself and $B(z)=z\underset{k=1}{\overset{6}{\prod }}\frac{z-a_{k}}{1-%
\overline{a_{k}}z}$ be a Blaschke product of degree $7$. Then $B\circ M=B$
if and only if $M(z)=c\frac{z-a_{l}}{1-\overline{a_{l}}z}$ with $\left\vert
a_{j}\right\vert =\left\vert a_{l}\right\vert $ for some $(0\leq j,l\leq 6),$
$c=-\frac{a_{j}}{a_{l}}$ and the equation%
\begin{equation}
\begin{array}{l}
a_{l}+5ca_{l}\left\vert a_{l}\right\vert ^{2}+8c^{2}a_{l}\left\vert
a_{l}\right\vert ^{2}+9c^{3}a_{l}\left\vert a_{l}\right\vert
^{2}+8c^{4}a_{l}\left\vert a_{l}\right\vert ^{2} \\
+5c^{5}a_{l}\left\vert a_{l}\right\vert ^{2}+6c^{2}a_{l}\left\vert
a_{l}\right\vert ^{4}+9c^{3}a_{l}\left\vert a_{l}\right\vert
^{4}+6c^{4}a_{l}\left\vert a_{l}\right\vert ^{4} \\
+c^{3}a_{l}\left\vert a_{l}\right\vert
^{6}+ca_{l}+c^{2}a_{l}+c^{3}a_{l}+c^{4}a_{l}+c^{5}a_{l}+c^{6}a_{l}=0%
\end{array}
\label{eqn14}
\end{equation}%
is satisfied by the non-zero zeros of $B$.
\end{corollary}

\begin{example}
Let $B(z)=z\underset{k=1}{\overset{6}{\prod }}\frac{z-a_{k}}{1-\overline{%
a_{k}}z}$ be a Blaschke product of degree $7$. From Proposition \ref{prop4},
we know that $Z(B)=\{0,M(0),...,M^{n-1}(0)\},$ so we can take%
\begin{eqnarray}
a_{1} &=&M\left( 0\right) ,  \label{eqn15} \\
a_{2} &=&M^{2}\left( 0\right) ,  \notag \\
a_{3} &=&M^{3}\left( 0\right) ,  \notag \\
a_{4} &=&M^{4}\left( 0\right) ,  \notag \\
a_{5} &=&M^{5}\left( 0\right) ,  \notag \\
a_{6} &=&M^{6}\left( 0\right) .  \notag
\end{eqnarray}%
Let $a_{l}=\frac{1}{2}.$ Using the equation $($\ref{eqn14}$),$ we obtain $%
c=0.217617-i0.976034$ and%
\begin{equation*}
M\left( z\right) =\frac{-\left( 0.217617-i0.976034\right) \left( 2z-1\right)
}{2-z}.
\end{equation*}%
Using the equation $($\ref{eqn15}$),$ we find%
\begin{eqnarray*}
a_{1} &=&-0.108809+i0.488017, \\
a_{2} &=&163605+i0.702141, \\
a_{3} &=&0.40682+i0.679542, \\
a_{4} &=&0.574725+i0.54495, \\
a_{5} &=&0.64971+i0.312482, \\
a_{6} &=&0.5.
\end{eqnarray*}%
Then we obtain $\left( B\circ M\right) \left( z\right) =B\left( z\right) $
for the points $z\in \mathbb{D}.$
\end{example}

Now we consider the canonical Blaschke products of degree $4$. At first,
from \cite{gassier}, we can give the following corollary for a Blaschke
product $B$ of degree $4$ .

\begin{corollary}
\label{corol1}$($See \cite{gassier}, page $204.)$ Let $G$ be a cyclic group
which is composed of the transformations $M$ such that $B\circ M=B$. Then we
have the following assertions$:$

$(i)$ If $B(z)=z^{4},$ $G$ is generated by $M(z)=iz,$ $z\in \overline{%
\mathbb{D}}.$

$(ii)$ If $Z(B)$ contains a non-zero point in $\mathbb{D}$, there are two
cases$:$

$(a)$%
\begin{equation*}
B(z)=z\frac{z-a_{1}}{1-\overline{a_{1}}z}\frac{z-a_{2}}{1-\overline{a_{2}}z}%
\frac{z-\frac{a_{1}-a_{2}}{1-\overline{a_{1}}a_{2}}}{1-z\left( \frac{%
\overline{a_{1}}-\overline{a_{2}}}{1-a_{1}\overline{a_{2}}}\right) },
\end{equation*}%
where $a_{1}$ or $a_{2}$ is non-equal to $0.$ In this case, $M(z)=-\frac{%
z-a_{1}}{1-\overline{a_{1}}z}$ and thus $M$ does not generate $G$ since the
degree of $M$ is equal to $2.$

$(b)$
\begin{equation*}
B(z)=z\frac{z-a_{1}}{1-\overline{a_{1}}z}\frac{z-M(a_{1})}{1-\overline{%
M(a_{1})}z}\frac{z-M^{2}(a_{1})}{1-\overline{M^{2}(a_{1})}z},
\end{equation*}%
with $a_{1}\in \mathbb{D}.$ In this case $M(z)=c\frac{z+\overline{c}a_{1}}{%
1+c\overline{a_{1}}z}$ generates $G$ with $\left\vert c\right\vert =1$ and $%
c+\overline{c}=-2\left\vert a_{1}\right\vert ^{2}.$
\end{corollary}

In this case there is a nice relation between decomposition of the finite
Blaschke products $B$ of order $4$ and the Poncelet curves associated with
them. From \cite{fujimura} and \cite{NYO3}, we know the following theorem.

\begin{theorem}
For any Blaschke product $B$ of order $4,$ $B$ is a composition of two
Blaschke products of degree $2$ if and only if the Poncelet curve $E$ of
this Blaschke product is an ellipse with the equation
\begin{equation*}
E:\left\vert z-a_{1}\right\vert +\left\vert z-a_{2}\right\vert =\left\vert 1-%
\overline{a_{1}}a_{2}\right\vert \sqrt{\frac{\left\vert a_{1}\right\vert
^{2}+\left\vert a_{2}\right\vert ^{2}-2}{\left\vert a_{1}\right\vert
^{2}\left\vert a_{2}\right\vert ^{2}-1}}.
\end{equation*}
\end{theorem}

It is also known that the decomposition of some Blaschke products $B$ of
degree $4$ is linked with the case that Poncelet curve of this Blaschke
product is an ellipse with a nice geometric property.

\begin{theorem}
\label{thm2}$($See \cite{NYO3}, Theorem $5.2$ on page $103.)$ Let $%
a_{1},a_{2}$ and $a_{3}$ be three distinct nonzero complex numbers with $%
\left\vert a_{i}\right\vert <1$ for $1\leq i\leq 3$ and $B(z)=z\underset{i=1}%
{\overset{3}{\prod }}\frac{z-a_{i}}{1-\overline{a_{i}}z}$ be a Blaschke
product of degree $4$ with the condition that one of its zeros, say $a_{1},$%
satisfies the following equation:%
\begin{equation*}
a_{1}+\overline{a_{1}}a_{2}a_{3}=a_{2}+a_{3}.
\end{equation*}%
Then the Poncelet curve associated with $B$ is the ellipse $E$ with the
equation%
\begin{equation*}
E:\left\vert z-a_{2}\right\vert +\left\vert z-a_{3}\right\vert =\left\vert 1-%
\overline{a_{2}}a_{3}\right\vert \sqrt{\frac{\left\vert a_{2}\right\vert
^{2}+\left\vert a_{3}\right\vert ^{2}-2}{\left\vert a_{2}\right\vert
^{2}\left\vert a_{3}\right\vert ^{2}-1}}.
\end{equation*}
\end{theorem}

Let $B(z)$ be given as in the statement of Theorem \ref{thm2}. For any $%
\lambda \in \partial \mathbb{D}$, let $z_{1},z_{2},z_{3}$ and $z_{4}$ be the
four distinct points satisfying $B(z_{1})=B(z_{2})=B(z_{3})=B(z_{4})=\lambda
$. Then the Poncelet curve associated with $B$ is an ellipse $E$ with foci $%
a_{2}$ and $a_{3}$ and the lines joining $z_{1},z_{3}$ and $z_{2},z_{4}$
passes through the point $a_{1}.$

\begin{example}
Let $a_{1}=\frac{2}{3},a_{2}=\frac{1}{2}-i\frac{1}{2},a_{3}=\frac{1}{2}+i%
\frac{1}{2}$ and $B(z)=z\underset{i=1}{\overset{3}{\prod }}\frac{z-a_{i}}{1-%
\overline{a_{i}}z}.$ The Poncelet curve associated with $B$ is an ellipse
with foci $a_{2}$ and $a_{3}$ $($see Figure \ref{fig:4}$).$
\end{example}

\begin{figure}[t]
\centering
\includegraphics[height=8cm, width=8cm]{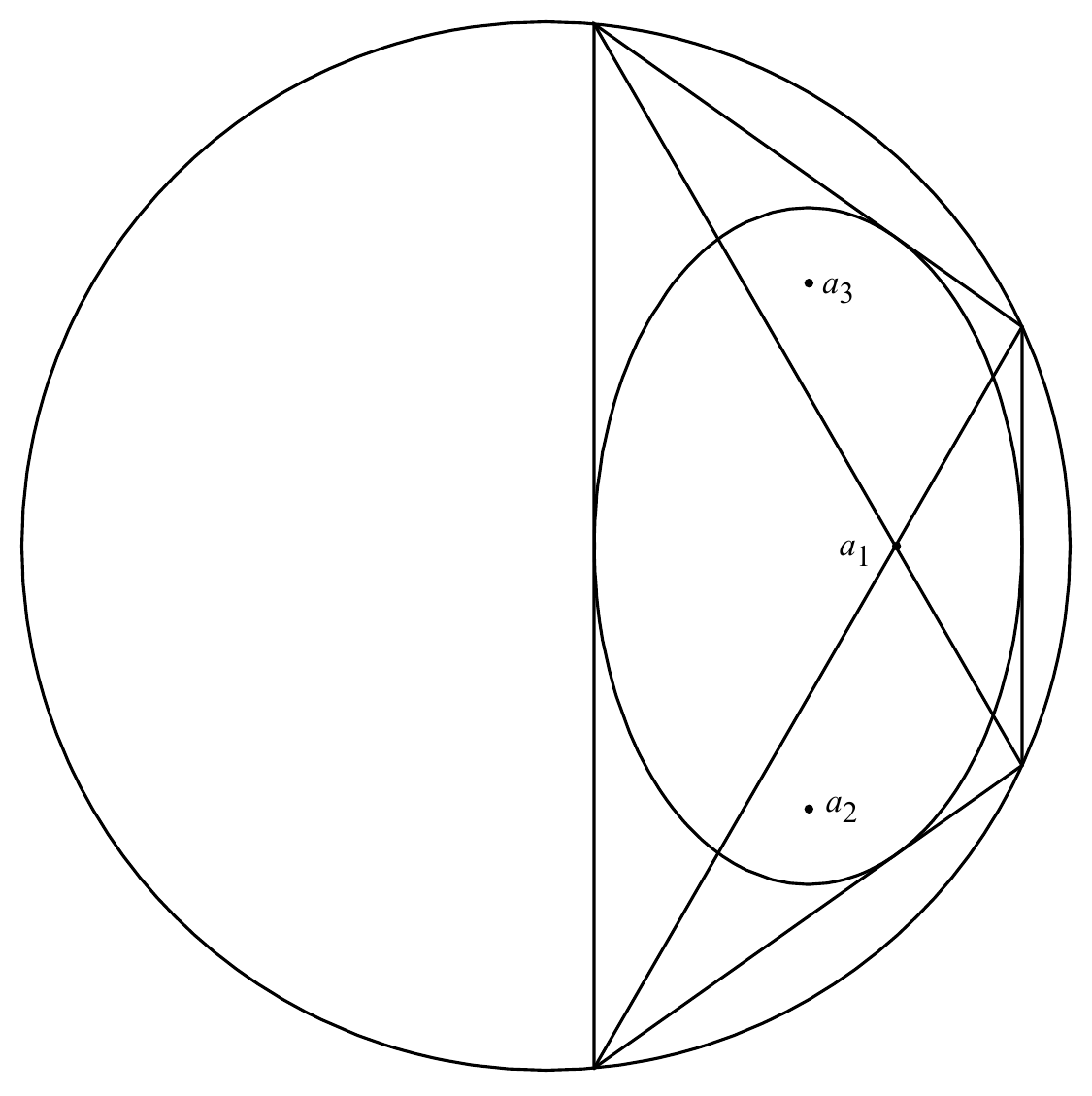}
\caption{{\protect\small {}}}
\label{fig:4}
\end{figure}

However decomposition of a Blaschke product $B$ is not always linked with
the Poncelet curve of the Blaschke product, as we have seen in the following
theorem.

\begin{theorem}
\label{thm3} $($See \cite{NYO1}, Theorem $4.2$ in page $69.)$ Let $%
a_{1},a_{2},...,a_{2n-1}$ be $2n-1$ distinct nonzero complex numbers with $%
\left\vert a_{k}\right\vert <1$ for $1\leq k\leq 2n-1$ and $B(z)=z\underset{%
k=1}{\overset{2n-1}{\prod }}\frac{z-a_{k}}{1-\overline{a_{k}}z}$ be a
Blaschke product of degree $2n$ with the condition that one of its zeros,
say $a_{1},$ satisfies the following equations$:$%
\begin{eqnarray*}
a_{1}+\overline{a_{1}}a_{2}a_{3} &=&a_{2}+a_{3}. \\
a_{1}+\overline{a_{1}}a_{4}a_{5} &=&a_{4}+a_{5}. \\
&&... \\
a_{1}+\overline{a_{1}}a_{2n-2}a_{2n-1} &=&a_{2n-2}+a_{2n-1}.
\end{eqnarray*}%
$(i)$ If $L$ is any line through the point $a_{1},$ then for the points $%
z_{1}$ and $z_{2}$ at which L intersects $\partial \mathbb{D}$, we have $%
B(z_{1})=B(z_{2}).$

$(ii)$ The unit circle $\partial \mathbb{D}$ and any circle through the
points $0$ and $\frac{1}{\overline{a}_{1}}$ have exactly two distinct
intersection points $z_{1}$ and $z_{2}.$ Then we have $B(z_{1})=B(z_{2})$
for these intersection points.
\end{theorem}

From the proof of Theorem \ref{thm3}, we know that Blaschke product $B$ can
be written as $B(z)=B_{2}\circ B_{1}(z)$ where $B_{1}(z)=z\frac{z-a_{1}}{1-%
\overline{a}_{1}z}$ and $B_{2}(z)=z\frac{%
(z+a_{2}a_{3})(z+a_{4}a_{5})...(z+a_{2n-2}a_{2n-1})}{(1+\overline{a_{2}a_{3}}%
z)(1+\overline{a_{4}a_{5}}z)...(1+\overline{a_{2n-2}a_{2n-1}}z)}.$ Now we
investigate under what conditions $B\circ M=B$ such that $M$ is different
from the identity for the Blaschke product given in Theorem \ref{thm3}. For $%
n=2$, we can give the following theorem.

\begin{theorem}
Let $a_{1},a_{2}$ and $a_{3}$ be three distinct nonzero complex numbers with
$\left\vert a_{k}\right\vert <1$ for $1\leq k\leq 3$ and $B(z)=z\underset{k=1%
}{\overset{3}{\prod }}\frac{z-a_{k}}{1-\overline{a_{k}}z}$ be a Blaschke
product of degree $4$ with the condition that one of its zeros, say $a_{1},$
satisfies the following equation$:$%
\begin{equation*}
a_{1}+\overline{a_{1}}a_{2}a_{3}=a_{2}+a_{3}.
\end{equation*}%
$(i)$ If $M(z)=-\frac{z-a_{1}}{1-\overline{a}_{1}z}$ then we have $%
B=B_{2}\circ B_{1}$ and $B=B\circ M$ only when $a_{3}=\frac{a_{1}-a_{2}}{1-%
\overline{a}_{1}a_{2}}.$

$(ii)$ Let $M(z)=c\frac{z+\overline{c}a_{1}}{1+c\overline{a}_{1}z}$ with the
conditions $c+\overline{c}=-2\left\vert a_{1}\right\vert ^{2},$ $%
a_{2}=M(a_{1})$ and $a_{3}=M^{2}(a_{1})$. If $a_{1}$ and $c$ with $%
\left\vert c\right\vert =1$ satisfy the following equation
\begin{equation}
a_{1}c+2a_{1}c^{2}\left\vert a_{1}\right\vert ^{2}+a_{1}c^{3}\left\vert
a_{1}\right\vert ^{2}+a_{1}-a_{1}\left\vert a_{1}\right\vert ^{2}=0,
\label{eqn9}
\end{equation}%
then we have $B=B_{2}\circ B_{1}$ and $B=B\circ M.$
\end{theorem}

\begin{proof}
We use the equation $a_{1}+\overline{a_{1}}a_{2}a_{3}=a_{2}+a_{3}$, Theorem %
\ref{thm3} and Corollary \ref{corol1}.

$(i)$ Let $M(z)=-\frac{z-a_{1}}{1-\overline{a}_{1}z}.$ Then by Corollary \ref%
{corol1}, the condition $B=B\circ M$ implies $a_{3}=\frac{a_{1}-a_{2}}{1-%
\overline{a}_{1}a_{2}}$.

$(ii)$ Let $M(z)=c\frac{z+\overline{c}a_{1}}{1+c\overline{a}_{1}z}.$ From
Corollary \ref{corol1}, it should be $a_{2}=M(a_{1}),$ $a_{3}=M^{2}(a_{1})$
and $c+\overline{c}=-2\left\vert a_{1}\right\vert ^{2}.$ Then we obtain
\begin{equation*}
a_{2}=\frac{a_{1}(1+c)\text{ }}{1+c\left\vert a_{1}\right\vert ^{2}}\text{
and }a_{3}=\frac{a_{1}(1-\left\vert a_{1}\right\vert ^{2})\text{ }}{%
1+2c\left\vert a_{1}\right\vert ^{2}+c^{2}\left\vert a_{1}\right\vert ^{2}}.
\end{equation*}%
If we substitute these values of $a_{2}$ and $a_{3}$ in the equation $a_{1}+%
\overline{a_{1}}a_{2}a_{3}=a_{2}+a_{3},$ we have the following equation
\begin{equation*}
a_{1}c+2a_{1}c^{2}\left\vert a_{1}\right\vert ^{2}+a_{1}c^{3}\left\vert
a_{1}\right\vert ^{2}+a_{1}-a_{1}\left\vert a_{1}\right\vert ^{2}=0.
\end{equation*}

Also in both cases we know that $B$ has a decomposition as $B=B_{2}\circ
B_{1}$ by Theorem \ref{thm2}. Thus the proof is completed.
\end{proof}

Now, we give two examples for the both cases of the above theorem.

\begin{example}
Let $B$ be a Blaschke product and $M$ be a M\"{o}bius transformation of the
following forms:%
\begin{equation*}
B(z)=z\frac{z-a_{1}}{1-\overline{a}_{1}z}\frac{z-a_{2}}{1-\overline{a}_{2}z}%
\frac{z-\left( \frac{a_{1}-a_{2}}{1-\overline{a}_{1}a_{2}}\right) }{%
1-z\left( \frac{\overline{a}_{1}-\overline{a}_{2}}{1-a_{1}\overline{a}_{2}}%
\right) }\text{ and }M(z)=\frac{-z+a_{1}}{1-\overline{a}_{1}z}.
\end{equation*}%
For $a_{1}=\frac{1}{2}$ and $a_{2}=\frac{1}{2}-\frac{i}{2}$ we obtain
\begin{equation*}
B(z)=\frac{z(z-\frac{1}{2})(z-\frac{1}{2}+\frac{i}{2})(z-\frac{1}{5}-\frac{3i%
}{5})}{(-\frac{1}{2}z+1)(1-z(\frac{1}{2}+\frac{i}{2}))(1-z(\frac{1}{5}-\frac{%
3i}{5}))}\text{ and }M(z)=\frac{-z+\frac{1}{2}}{1-\frac{1}{2}z}.
\end{equation*}%
Then we find $(B\circ M)(z)=B(z)$ and $B(z)=(B_{2}\circ B_{1})(z)$ for the
points $z\in \mathbb{D}.$
\end{example}

\begin{example}
\ Let $B$ be a Blaschke product and $M$ be a M\"{o}bius transformation of
the following forms:%
\begin{equation*}
B(z)=z\frac{z-a_{1}}{1-\overline{a_{1}}z}\frac{z-M(a_{1})}{1-\overline{%
M(a_{1})}z}\frac{z-M^{2}(a_{1})}{1-\overline{M^{2}(a_{1})}z}\text{ and }%
M(z)=c\frac{z+\overline{c}a_{1}}{1+c\overline{a_{1}}z}.
\end{equation*}%
For $a_{1}=\frac{2}{3},$ solving the equation $a_{1}c+2a_{1}c^{2}\left\vert
a_{1}\right\vert ^{2}+a_{1}c^{3}\left\vert a_{1}\right\vert
^{2}+a_{1}-a_{1}\left\vert a_{1}\right\vert ^{2}=0$ we find $c=-1.$ Then we
have $B$ and $M$ of the following forms:
\begin{equation*}
B(z)=\frac{z^{2}(-\frac{2}{3}+z)^{2}}{(1-\frac{2}{3}z)^{2}}\text{ and }M(z)=-%
\frac{-\frac{2}{3}+z}{1-\frac{2}{3}z}
\end{equation*}%
Then we find $(B\circ M)(z)=B(z)$ and $B(z)=(B_{2}\circ B_{1})(z)$ for the
points $z\in \mathbb{D}.$
\end{example}

From the above discussions, we can say that decomposition of a finite
Blaschke product $B$ is linked with its zeros. But for a finite Blaschke
product $B$ of degree $4$, this case is also linked with the Poncelet curve
of $B.$

Using Theorem \ref{thm6} and Theorem \ref{thm3}, for the Blaschke product of
degree $2n$, we give the following result.

\begin{corollary}
\label{thm8}Let $a_{1},a_{2},...,a_{2n-1}$ be $2n-1$ distinct nonzero
complex numbers with $\left\vert a_{k}\right\vert <1$ for $1\leq k\leq 2n-1$
and $B\left( z\right) =z\underset{k=1}{\overset{2n-1}{\prod }}\frac{z-a_{k}}{%
1-\overline{a_{k}}z}$ be a Blaschke products of degree $2n$ with the
condition that one of its zeros, say $a_{1}$, satisfies the following
equations:%
\begin{eqnarray*}
a_{1}+\overline{a_{1}}a_{2}a_{3} &=&a_{2}+a_{3} \\
a_{1}+\overline{a_{1}}a_{4}a_{5} &=&a_{4}+a_{5} \\
&&\cdots \\
a_{1}+\overline{a_{1}}a_{2n-2}a_{2n-1} &=&a_{2n-2}+a_{2n-1}.
\end{eqnarray*}%
Let $M\left( z\right) =c\frac{z-a_{2n-1}}{1-z\overline{a_{2n-1}}}$ with the
conditions $\left\vert c\right\vert =1,$ $M^{2n-1}\left( 0\right)
-a_{2n-1}=0,$ $a_{1}=M\left( 0\right) ,a_{2}=M^{2}\left( 0\right)
,...,a_{2n-1}=M^{2n-1}\left( 0\right) $. If $a_{2n-1}$ and $c$ satisfy
following equations:%
\begin{eqnarray*}
M\left( 0\right) +\overline{M\left( 0\right) }M^{2}\left( 0\right)
M^{3}\left( 0\right) &=&M^{2}\left( 0\right) +M^{3}\left( 0\right) \\
M\left( 0\right) +\overline{M\left( 0\right) }M^{4}\left( 0\right)
M^{5}\left( 0\right) &=&M^{4}\left( 0\right) +M^{5}\left( 0\right) \\
&&\cdots \\
M\left( 0\right) +\overline{M\left( 0\right) }M^{2n-2}\left( 0\right)
M^{2n-1}\left( 0\right) &=&M^{2n-2}\left( 0\right) +M^{2n-1}\left( 0\right)
\end{eqnarray*}%
Then we have $B=B_{2}\circ B_{1}$ and $B\circ M=B$.
\end{corollary}

\begin{proof}
The proof is obvious from Theorem \ref{thm6} and Theorem \ref{thm3}.
\end{proof}

\begin{example}
Let $B$ be a Blaschke product and $M$ be a M\"{o}bius transformation of the
following forms:%
\begin{equation*}
B(z)=z\frac{z-M(0)}{1-\overline{M(0)}z}\frac{z-M^{2}(0)}{1-\overline{M^{2}(0)%
}z}\frac{z-M^{3}(0)}{1-\overline{M^{3}(0)}z}\frac{z-M^{4}(0)}{1-\overline{%
M^{4}(0)}z}\text{ }\frac{z-M^{5}(0)}{1-\overline{M^{5}(0)}z}
\end{equation*}%
and
\begin{equation*}
M(z)=c\frac{z-a_{5}}{1-\overline{a_{5}}z}.
\end{equation*}%
From Corollary \ref{thm8}, it should be
\begin{equation}
M\left( 0\right) +\overline{M\left( 0\right) }M^{2}\left( 0\right)
M^{3}\left( 0\right) =M^{2}\left( 0\right) +M^{3}\left( 0\right)
\label{eqn16}
\end{equation}%
and%
\begin{equation}
M\left( 0\right) +\overline{M\left( 0\right) }M^{4}\left( 0\right)
M^{5}\left( 0\right) =M^{4}\left( 0\right) +M^{5}\left( 0\right) \text{.}
\label{eqn17}
\end{equation}%
We know that $a_{1}=M\left( 0\right) =-ca_{5},a_{2}=M^{2}\left( 0\right)
=-ca_{5}\frac{\left( 1+c\right) }{1+c\left\vert a_{5}\right\vert ^{2}}%
,a_{3}=M^{3}\left( 0\right) =-ca_{5}\frac{\left( 1+c+c^{2}+c\left\vert
a_{5}\right\vert ^{2}\right) }{1+2c\left\vert a_{5}\right\vert
^{2}+c^{2}\left\vert a_{5}\right\vert ^{2}}$ and $a_{4}=M^{4}\left( 0\right)
=-ca_{5}\frac{\left( 1+c\right) \left( 1+c^{2}+2c\left\vert a_{5}\right\vert
^{2}\right) }{1+c\left\vert a_{5}\right\vert ^{2}\left(
3+2c+c^{2}+c\left\vert a_{5}\right\vert ^{2}\right) }.$ Writing these values
in the equations $($\ref{eqn16}$)$ and $($\ref{eqn17}$)$, we have
\begin{equation}
\begin{array}{l}
ca_{5}+2c^{2}a_{5}+c^{3}a_{5}+c^{3}a_{5}\left\vert a_{5}\right\vert
^{2}+c^{4}a_{5}\left\vert a_{5}\right\vert ^{2}-2c^{3}a_{5}\left\vert
a_{5}\right\vert ^{4} \\
-c^{4}a_{5}\left\vert a_{5}\right\vert ^{4}-ca_{5}\left\vert
a_{5}\right\vert ^{2}-c^{2}a_{5}\left\vert a_{5}\right\vert
^{2}-c^{2}a_{5}\left\vert a_{5}\right\vert ^{4}=0%
\end{array}
\label{eqn18}
\end{equation}

and%
\begin{equation}
\begin{array}{l}
a_{5}-c^{2}a_{5}-c^{3}a_{5}-c^{4}a_{5}+3ca_{5}\left\vert a_{5}\right\vert
^{2}+c^{2}a_{5}\left\vert a_{5}\right\vert ^{4}+2c^{2}a_{5}\left\vert
a_{5}\right\vert ^{2}+c^{4}a_{5}\left\vert a_{5}\right\vert ^{2} \\
+c^{3}a_{5}\left\vert a_{5}\right\vert ^{4}-ca_{5}\left\vert
a_{5}\right\vert ^{2}-2c^{2}a_{5}\left\vert a_{5}\right\vert
^{4}-a_{5}\left\vert a_{5}\right\vert ^{2}-2ca_{5}\left\vert
a_{5}\right\vert ^{4}=0\text{.}%
\end{array}
\label{eqn19}
\end{equation}%
For $a_{5}=\frac{1}{2},$ solving the equations $($\ref{eqn18}$)$ and $($\ref%
{eqn19}$)$ we find $c=-1.$ Then we have $B$ and $M$ of the following forms:%
\begin{equation*}
B(z)=z\frac{\left( 2z-1\right) \left( z-0.5\right) ^{2}\left( z-1.4803\times
10^{-16}\right) \left( z-7.40149\times 10^{-17}\right) }{\left( 2-z\right)
\left( 1-0.5z\right) ^{2}\left( 1-1.4803\times 10^{-16}z\right) \left(
1-7.40149\times 10^{-17}z\right) }
\end{equation*}%
and
\begin{equation*}
M(z)=\frac{2z-1}{z-2}.
\end{equation*}%
Then for the points $z\in \mathbb{D}$, we find
\begin{equation*}
(B\circ M)(z)=B(z)\text{ and }B(z)=(B_{2}\circ B_{1})(z)
\end{equation*}%
where
\begin{equation*}
B_{1}(z)=z\frac{z-0.5}{1-0.5z}\text{ and }B_{2}(z)=z\frac{z+3.70074\times
10^{-17}}{1+3.70074\times 10^{-17}z}\frac{z+7.40149\times 10^{-17}}{%
1+7.40149\times 10^{-17}z}.
\end{equation*}
\end{example}

\begin{corollary}
\label{corol2}Let $a_{1},a_{2},...,a_{3n-1}$ be $3n-1$ distinct nonzero
complex numbers with $\left\vert a_{k}\right\vert <1$ for $1\leq k\leq 3n-1$
and $B\left( z\right) =z\underset{k=1}{\overset{3n-1}{\prod }}\frac{z-a_{k}}{%
1-\overline{a_{k}}z}$ be a Blaschke products of degree $3n$ with the
condition that one of its zeros, say $a_{1}$, satisfies the following
equations:%
\begin{eqnarray*}
a_{1}+a_{2}+a_{3}a_{4}a_{5}\overline{a_{1}a_{2}} &=&a_{3}+a_{4}+a_{5}, \\
a_{1}a_{2}+a_{3}a_{4}a_{5}\left( \overline{a_{1}}+\overline{a_{2}}\right)
&=&a_{3}a_{4}+a_{3}a_{5}+a_{4}a_{5}, \\
&&\cdots \\
a_{1}+a_{2}+a_{3n-3}a_{3n-2}a_{3n-1}\overline{a_{1}a_{2}}
&=&a_{3n-3}+a_{3n-2}+a_{3n-1}, \\
a_{1}a_{2}+a_{3n-3}a_{3n-2}a_{3n-1}\left( \overline{a_{1}}+\overline{a_{2}}%
\right) &=&a_{3n-3}a_{3n-2}+a_{3n-3}a_{3n-1}+a_{3n-2}a_{3n-1}.
\end{eqnarray*}%
Let $M\left( z\right) =c\frac{z-a_{3n-1}}{1-z\overline{a_{3n-1}}}$ with the
conditions $\left\vert c\right\vert =1,$ $M^{3n-1}\left( 0\right)
-a_{3n-1}=0,$ $a_{1}=M\left( 0\right) ,a_{2}=M^{2}\left( 0\right)
,...,a_{3n-1}=M^{3n-1}\left( 0\right) $. If $a_{3n-1}$ and $c$ satisfy
following equations:%
\begin{equation*}
\begin{array}{l}
M\left( 0\right) +M^{2}\left( 0\right) +M^{3}\left( 0\right) M^{4}\left(
0\right) M^{5}\left( 0\right) \overline{M\left( 0\right) M^{2}\left(
0\right) } \\
=M^{3}\left( 0\right) +M^{4}\left( 0\right) +M^{5}\left( 0\right) , \\
M\left( 0\right) M^{2}\left( 0\right) +M^{3}\left( 0\right) M^{4}\left(
0\right) M^{5}\left( 0\right) \left( \overline{M\left( 0\right) }+\overline{%
M^{2}\left( 0\right) }\right) \\
=M^{3}\left( 0\right) M^{4}\left( 0\right) +M^{3}\left( 0\right) M^{5}\left(
0\right) +M^{4}\left( 0\right) M^{5}\left( 0\right) , \\
\cdots \\
M\left( 0\right) +M^{2}\left( 0\right) +M^{3n-3}\left( 0\right)
M^{3n-2}\left( 0\right) M^{3n-1}\left( 0\right) \overline{M\left( 0\right)
M^{2}\left( 0\right) } \\
=M^{3n-3}\left( 0\right) +M^{3n-2}\left( 0\right) +M^{3n-1}\left( 0\right) ,
\\
M\left( 0\right) M^{2}\left( 0\right) +M^{3n-3}\left( 0\right)
M^{3n-2}\left( 0\right) M^{3n-1}\left( 0\right) \left( \overline{M\left(
0\right) }+\overline{M^{2}\left( 0\right) }\right) \\
=M^{3n-3}\left( 0\right) M^{3n-2}\left( 0\right) +M^{3n-3}\left( 0\right)
M^{3n-1}\left( 0\right) +M^{3n-2}\left( 0\right) M^{3n-1}\left( 0\right) .%
\end{array}%
\end{equation*}%
Then we have $B=B_{2}\circ B_{1}$ and $B\circ M=B$.
\end{corollary}

\begin{proof}
Let $a_{1},a_{2},...,a_{3n-1}$ be $3n-1$ distinct nonzero complex numbers
with $\left\vert a_{k}\right\vert <1$ for $1\leq k\leq 3n-1$ and $B\left(
z\right) =z\underset{k=1}{\overset{3n-1}{\prod }}\frac{z-a_{k}}{1-\overline{%
a_{k}}z}$ be a Blaschke product of degree $3n$ with the condition that two
of its zeros, say $a_{1}$ and $a_{2},$ satisfies the following equations:
\begin{eqnarray*}
a_{1}+a_{2}+a_{3}a_{4}a_{5}\overline{a_{1}a_{2}} &=&a_{3}+a_{4}+a_{5}, \\
a_{1}a_{2}+a_{3}a_{4}a_{5}\left( \overline{a_{1}}+\overline{a_{2}}\right)
&=&a_{3}a_{4}+a_{3}a_{5}+a_{4}a_{5}, \\
a_{1}+a_{2}+a_{6}a_{7}a_{8}\overline{a_{1}a_{2}} &=&a_{6}+a_{7}+a_{8}, \\
a_{1}a_{2}+a_{6}a_{7}a_{8}\left( \overline{a_{1}}+\overline{a_{2}}\right)
&=&a_{6}a_{7}+a_{6}a_{8}+a_{7}a_{8}, \\
&&\cdots \\
a_{1}+a_{2}+a_{3n-3}a_{3n-2}a_{3n-1}\overline{a_{1}a_{2}}
&=&a_{3n-3}+a_{3n-2}+a_{3n-1}, \\
a_{1}a_{2}+a_{3n-3}a_{3n-2}a_{3n-1}\left( \overline{a_{1}}+\overline{a_{2}}%
\right) &=&a_{3n-3}a_{3n-2}+a_{3n-3}a_{3n-1}+a_{3n-2}a_{3n-1}.
\end{eqnarray*}%
By Theorem $4.4$ on page $71$ in \cite{NYO1}, we know that B$\left( z\right)
$ can be written as a composition of two Blaschke products of degree $3$ and
$n$ as $B\left( z\right) =\left( B_{2}\circ B_{1}\right) \left( z\right) $
where
\begin{equation*}
B_{1}\left( z\right) =\frac{z\left( z-a_{1}\right) (z-a_{2})}{\left( 1-%
\overline{a_{1}}z\right) \left( 1-\overline{a_{2}}z\right) }
\end{equation*}%
and%
\begin{equation*}
B_{2}\left( z\right) =\frac{z\left( z-a_{3}a_{4}a_{5}\right) \left(
z-a_{6}a_{7}a_{8}\right) ...\left( z-a_{3n-3}a_{3n-2}a_{3n-1}\right) }{%
\left( 1-\overline{a_{3}a_{4}a_{5}}z\right) \left( 1-\overline{%
a_{6}a_{7}a_{8}}z\right) ...\left( 1-\overline{a_{3n-3}a_{3n-2}a_{3n-1}}%
z\right) }
\end{equation*}%
Then, the rest of the proof is clear from Theorem \ref{thm6}.
\end{proof}

\begin{example}
Let $B$ be a Blaschke product and $M$ be a M\"{o}bius transformation of the
following forms:%
\begin{equation*}
B(z)=z\frac{z-M(0)}{1-\overline{M(0)}z}\frac{z-M^{2}(0)}{1-\overline{M^{2}(0)%
}z}\frac{z-M^{3}(0)}{1-\overline{M^{3}(0)}z}\frac{z-M^{4}(0)}{1-\overline{%
M^{4}(0)}z}\text{ }\frac{z-M^{5}(0)}{1-\overline{M^{5}(0)}z}
\end{equation*}%
and%
\begin{equation*}
M(z)=c\frac{z-a_{5}}{1-\overline{a_{5}}z}.
\end{equation*}%
From Corollary \ref{corol2}, it should be%
\begin{equation}
\begin{array}{c}
M\left( 0\right) +M^{2}\left( 0\right) +M^{3}\left( 0\right) M^{4}\left(
0\right) M^{5}\left( 0\right) \overline{M\left( 0\right) M^{2}\left(
0\right) } \\
=M^{3}\left( 0\right) +M^{4}\left( 0\right) +M^{5}\left( 0\right)%
\end{array}
\label{eqn20}
\end{equation}%
and%
\begin{equation}
\begin{array}{c}
M\left( 0\right) M^{2}\left( 0\right) +M^{3}\left( 0\right) M^{4}\left(
0\right) M^{5}\left( 0\right) \left( \overline{M\left( 0\right) }+\overline{%
M^{2}\left( 0\right) }\right) \\
=M^{3}\left( 0\right) M^{4}\left( 0\right) +M^{3}\left( 0\right) M^{5}\left(
0\right) +M^{4}\left( 0\right) M^{5}\left( 0\right) .%
\end{array}
\label{eqn21}
\end{equation}%
We know that $a_{1}=M\left( 0\right) =-ca_{5},a_{2}=M^{2}\left( 0\right)
=-ca_{5}\frac{\left( 1+c\right) }{1+c\left\vert a_{5}\right\vert ^{2}}%
,a_{3}=M^{3}\left( 0\right) =-ca_{5}\frac{\left( 1+c+c^{2}+c\left\vert
a_{5}\right\vert ^{2}\right) }{1+2c\left\vert a_{5}\right\vert
^{2}+c^{2}\left\vert a_{5}\right\vert ^{2}}$ and $a_{4}=M^{4}\left( 0\right)
=-ca_{5}\frac{\left( 1+c\right) \left( 1+c^{2}+2c\left\vert a_{5}\right\vert
^{2}\right) }{1+c\left\vert a_{5}\right\vert ^{2}\left(
3+2c+c^{2}+c\left\vert a_{5}\right\vert ^{2}\right) }.$ Writing these values
in the equations $($\ref{eqn20}$)$ and $($\ref{eqn21}$)$, we have%

\begin{equation}
\footnotesize
\begin{array}{l}
-ca_{5}\left( 1+c\left\vert a_{5}\right\vert ^{2}\right) \left(
1+2c\left\vert a_{5}\right\vert ^{2}+c^{2}\left\vert a_{5}\right\vert
^{2}\right) \left( 1+c\left\vert a_{5}\right\vert ^{2}\left(
3+2c+c^{2}+c\left\vert a_{5}\right\vert ^{2}\right) \right) \left( 1+%
\overline{c}\left\vert a_{5}\right\vert ^{2}\right) \\
-ca_{5}\left( 1+c\right) \left( 1+2c\left\vert a_{5}\right\vert
^{2}+c^{2}\left\vert a_{5}\right\vert ^{2}\right) \left( 1+c\left\vert
a_{5}\right\vert ^{2}\left( 3+2c+c^{2}+c\left\vert a_{5}\right\vert
^{2}\right) \right) \left( 1+\overline{c}\left\vert a_{5}\right\vert
^{2}\right) \\
+a_{5}\left\vert a_{5}\right\vert ^{4}\left( 1+c+c^{2}+c\left\vert
a_{5}\right\vert ^{2}\right) \left( 1+c\right) \left( 1+c^{2}+2c\left\vert
a_{5}\right\vert ^{2}\right) \left( 1+\overline{c}\right) \left(
1+c\left\vert a_{5}\right\vert ^{2}\right) \\
+ca_{5}\left( 1+c+c^{2}+c\left\vert a_{5}\right\vert ^{2}\right) \left(
1+c\left\vert a_{5}\right\vert ^{2}\right) \left( 1+c\left\vert
a_{5}\right\vert ^{2}\left( 3+2c+c^{2}+c\left\vert a_{5}\right\vert
^{2}\right) \right) \left( 1+\overline{c}\left\vert a_{5}\right\vert
^{2}\right) \\
+ca_{5}\left( 1+c\right) \left( 1+c^{2}+2c\left\vert a_{5}\right\vert
^{2}\right) \left( 1+c\left\vert a_{5}\right\vert \right) ^{2}\left(
1+2c\left\vert a_{5}\right\vert ^{2}+c^{2}\left\vert a_{5}\right\vert
^{2}\right) \left( 1+\overline{c}\left\vert a_{5}\right\vert \right) ^{2} \\
-a_{5}\left( 1+c\left\vert a_{5}\right\vert ^{2}\right) \left(
1+2c\left\vert a_{5}\right\vert ^{2}+c^{2}\left\vert a_{5}\right\vert
^{2}\right) \left( 1+c\left\vert a_{5}\right\vert ^{2}\left(
3+2c+c^{2}+c\left\vert a_{5}\right\vert ^{2}\right) \right) \left( 1+%
\overline{c}\left\vert a_{5}\right\vert ^{2}\right) =0.%
\end{array}
\label{eqn22}
\end{equation}%
and%
\begin{equation}
\footnotesize
\begin{array}{l}
c^{2}a_{5}^{2}\left( 1+c\right) \left( 1+2c\left\vert a_{5}\right\vert
^{2}+c^{2}\left\vert a_{5}\right\vert ^{2}\right) \left( 1+c\left\vert
a_{5}\right\vert ^{2}\left( 3+2c+c^{2}+c\left\vert a_{5}\right\vert
^{2}\right) \right) \left( 1+\overline{c}\left\vert a_{5}\right\vert
^{2}\right) \left( 1+\overline{c}\left\vert a_{5}\right\vert ^{2}\right) \\
+c^{2}a_{5}^{3}\left( 1+c+c^{2}+c\left\vert a_{5}\right\vert ^{2}\right)
\left( 1+c\right) \left( 1+c^{2}+2c\left\vert a_{5}\right\vert ^{2}\right)
\left( -\overline{c}\overline{a_{5}}\left( 1+\overline{c}\left\vert
a_{5}\right\vert ^{2}\right) -\overline{c}\overline{a_{5}}\left( 1+\overline{%
c}\right) \right) \left( 1+c\left\vert a_{5}\right\vert ^{2}\right) \\
-c^{2}a_{5}^{2}\left( 1+c+c^{2}+c\left\vert a_{5}\right\vert ^{2}\right)
\left( 1+c\right) \left( 1+c^{2}+2c\left\vert a_{5}\right\vert ^{2}\right)
\left( 1+\overline{c}\left\vert a_{5}\right\vert ^{2}\right) \left(
1+c\left\vert a_{5}\right\vert ^{2}\right) \\
+ca_{5}^{2}\left( 1+c+c^{2}+c\left\vert a_{5}\right\vert ^{2}\right) \left(
1+c\left\vert a_{5}\right\vert ^{2}\right) \left( 1+c\left\vert
a_{5}\right\vert ^{2}\left( 3+2c+c^{2}+c\left\vert a_{5}\right\vert
^{2}\right) \right) \left( 1+\overline{c}\left\vert a_{5}\right\vert
^{2}\right) \\
+ca_{5}^{2}\left( 1+c\right) \left( 1+c^{2}+2c\left\vert a_{5}\right\vert
^{2}\right) \left( 1+c\left\vert a_{5}\right\vert ^{2}\right) \left(
1+2c\left\vert a_{5}\right\vert ^{2}+c^{2}\left\vert a_{5}\right\vert
^{2}\right) \left( 1+\overline{c}\left\vert a_{5}\right\vert ^{2}\right) =0.%
\end{array}
\label{eqn23}
\end{equation}

For $a_{5}=\frac{1}{2},$ solving the equations $($\ref{eqn22}$)$ and $($\ref%
{eqn23}$)$ we find $c=-0.625+i0.780625$ Then we have $B$ and $M$ of the
following forms:%

\begin{equation}
\tiny
B(z)=z\frac{\left( 2z-1\right) \left( z-0.5+6.39697\times 10^{-11}i\right)
\left( z-0.3125+0.390312i\right) ^{2}\left( z-2.51094\times
10^{-11}-1.05107\times 10^{-10}i\right) }{\left( 2-z\right) \left( 1-z\left(
0.5+6.39697\times 10^{-11}i\right) \right) \left( 1-z\left(
0.3125+0.390312i\right) \right) ^{2}\left( 1-z\left( 2.51094\times
10^{-11}-1.05107\times 10^{-10}i\right) \right) }\text{ }
\end{equation}%

and%
\begin{equation*}
M(z)=\left( 0.625-0.780625i\right) \frac{1-2z}{2-z}.\text{ }
\end{equation*}%
Then for the points $z\in \mathbb{D}$, we find%
\begin{equation*}
(B\circ M)(z)=B(z)\text{ and }B(z)=(B_{2}\circ B_{1})(z)
\end{equation*}%
where%
\begin{equation*}
B_{1}(z)=z\frac{\left( z-0.5+6.39697\times 10^{-11}i\right) \left(
z-0.3125+0.390312i\right) }{\left( 1-z\left( 0.5+6.39697\times
10^{-11}i\right) \right) \left( 1-z\left( 0.3125+0.390312i\right) \right) }%
\text{ }
\end{equation*}%
and%
\begin{equation*}
B_{2}(z)=z\frac{z-2.44357\times 10^{-11}-1.15227\times 10^{-11}i}{1-z\left(
2.44357\times 10^{-11}-1.15227\times 10^{-11}i\right) }\text{.}
\end{equation*}
\end{example}

\end{document}